\documentclass[11pt]{amsart}
\usepackage{amssymb, amstext, amscd, amsmath}
\usepackage[all]{xy}

\theoremstyle{plain}
\newtheorem{theorem}{Theorem}[section]
\newtheorem{corollary}[theorem]{Corollary}
\newtheorem{proposition}[theorem]{Proposition}
\newtheorem{lemma}[theorem]{Lemma}

\newtheorem{question}{Question}

\theoremstyle{definition}
\newtheorem{definition}[theorem]{Definition}
\newtheorem{example}[theorem]{Example}

\theoremstyle{definition}

\newtheorem*{acknow}{Acknowledgements}
\newtheorem*{examples}{Examples}

\makeatletter
\def\@cite#1#2{{\m@th\upshape\bfseries%
[{#1\if@tempswa{\m@th\upshape\mdseries, #2}\fi}]}} \makeatother

\newcommand{\bbC}{{\mathbb{C}}}

\newcommand{\bbF}{{\mathbb{F}}}

\newcommand{\bbN}{{\mathbb{N}}}

\newcommand{\bbZ}{{\mathbb{Z}}}

\renewcommand{\H}{{\mathcal{H}}}

\renewcommand{\O}{{\mathcal{O}}}

\newcommand{\R}{{\mathcal{R}}}
\renewcommand{\S}{{\mathcal{S}}}
\newcommand{\T}{{\mathcal{T}}}

\renewcommand{\epsilon}{\varepsilon}
\renewcommand{\phi}{\varphi}
\newcommand{\upchi}{{\raise.35ex\hbox{\ensuremath{\chi}}}}

\def\al{\alpha}

\def\be{\beta}

\def\de{\delta}

\def\ga{\gamma}
\def\eps{\varepsilon}

\newcommand{\foral}{\text{ for all }}

\newcommand{\alg}{\operatorname{alg}}

\newcommand{\id}{{\operatorname{id}}}

\newcommand\ad{\operatorname{ad}}

\newcommand{\ca}{\mathrm{C}^*}

\newcommand{\ssim}{\stackrel{s}{\sim}}

\newcommand{\nor}[1]{\left\Vert #1\right\Vert}

\begin{document}

\title{Isomorphism Invariants for Multivariable $\ca$-Dynamics}

\author[E.T.A. Kakariadis]{Evgenios T.A. Kakariadis}
\address{School of Mathematics and Statistics\\Newcastle University\\Newcastle upon Tyne\\ NE1 7RU\\UK}
\email{evgenios.kakariadis@ncl.ac.uk}

\author[E.G. Katsoulis]{Elias~G.~Katsoulis}
\address{Department of Mathematics \\East Carolina University\\ Greenville, NC 27858-4353\\USA}
\email{katsoulise@ecu.edu}

\thanks{First author partially supported by the Fields Institute for Research in the Mathematical Sciences}

\maketitle

\begin{abstract}
To a given multivariable $\ca$-dynamical system $(A, \al)$ consisting of $*$-automorphisms, we associate a family of operator algebras $\alg(A, \al)$, which includes as specific examples the tensor algebra and the semicrossed product. It is shown that if two such operator algebras $\alg(A, \al)$ and $\alg(B, \be)$ are isometrically isomorphic, then the induced dynamical systems $(\hat{A}, \hat{\al})$ and $(\hat{B}, \hat{\be})$ on the Fell spectra are piecewise conjugate, in the sense of Davidson and Katsoulis.

In the course of proving the above theorem we obtain several results of independent interest. If $\alg(A, \al)$ and $\alg(B, \be)$ are isometrically isomorphic, then the associated correspondences $X_{(A, \al)}$ and $X_{(B, \be)}$ are unitarily equivalent. In particular, the tensor algebras are isometrically isomorphic if and only if the associated correspondences are unitarily equivalent. Furthermore, isomorphism of semicrossed products implies isomorphism of the associated tensor algebras.

In the case of multivariable systems acting on $\ca$-algebras with trivial center, unitary equivalence of the associated correspondences reduces to outer conjugacy of the systems. This provides a complete invariant for isometric isomorphisms between semicrossed products as well.
\end{abstract}

\section{Introduction}

Given a multivariable $\ca$-dynamical system $(A, \al)$, consisting of $n_\al$ $*$-endomorphisms $\al \equiv (\al_1, \al_2, \dots ,\al_{n_{\al}})$ of a $\ca$-algebra $A$, one can associate various non-selfadjoint operator algebras $\alg(A,\al)$ that share certain common properties. Notable examples are the semicrossed product $A \times_\al \bbF_n^+$ and the tensor algebra $\T_{(A,\al)}^+$ of the associated $\ca$-correspondence $X_{(A,\al)}$. In one form or another, algebras of this type have been investigated by various authors over the last forty years \cite{Arv1, Arv, ArvJ, DKak, DK08, DKsimple, DK, DKM, DRS, HadH, MS3, Pet, Pet2, Pow}, beginning with the seminal work of Arveson~\cite{Arv1}. The central theme in these investigations has been to identify to what extend the dynamics of the system affect the classification of the associated operator algebras. This is also the main theme of the present paper.

For a single variable system $(A, \al)$ (when $n_\al=1$), the tensor algebra $\T_{(A,\al)}^+$ coincides with the semicrossed product $A \times_\al \bbZ^+$, better known as Peters' semicrossed product. In a series of papers, Arveson \cite{Arv}, Arveson and Josephson \cite{ArvJ}, Peters \cite{Pet2}, Hadwin and Hoover \cite{HadH} and Power~\cite{Pow} proved that topological conjugacy is a complete invariant for algebraic isomorphism between Peters' semicrossed products, provided that the ambient $\ca$-algebra is commutative and the action satisfies various extra conditions. These extra conditions were removed by Davidson and Katsoulis \cite{DK08} thus proving that topological conjugacy is a complete invariant for algebraic isomorphism between Petrer's semicrossed products over commutative $\ca$-algebras. For more general $\ca$-algebras  Muhly and Solel~\cite{MS} proposed the study of isometric isomorphisms between Peters' semicrossed products (actually between arbitrary tensor algebras) as a more manageable  problem. In \cite{MS}, they established that outer conjugacy is a complete invariant for isometric isomorphisms when the actions are automorphic and their Connes' spectrum is full. In \cite{DKsimple} Davidson and Katsoulis disposed the condition on the spectrum, provided that the ambient $\ca$-algebras are simple. Recently, Davidson and Kakariadis \cite{DKak} established outer conjugacy as a complete invariant for isometric isomorphisms of semicrossed products over arbitrary $\ca$-algebras, provided that the actions are injective and/or onto and in several other cases. The techniques from \cite{DKak} will be of importance to us.

The study of isomorphisms between operator algebras associated with multivariable systems is much more recent. Davidson and Katsoulis \cite[Theorem 3.22]{DK} have shown that piecewise conjugacy is an invariant for algebraic isomorphisms between tensor algebras or semicrossed products associated with multivariable systems over commutative $\ca$-algebras. Piecewise conjugacy turns out to be a complete invariant for algebraic isomorphisms between various classes of tensor algebras \cite[Theorem 3.25]{DK}. The general case however remains an important open problem. In general, the following two problems are open for arbitrary multivariable $\ca$-dynamical systems and are studied here:
\begin{itemize}
\item[(i)] Identify a complete invariant for isomorphisms between operator algebras associated with multivariable systems.
\item[(ii)] Develop a notion of piecewise conjugacy for multivariable systems that is an isomorphism invariant between the associated operator algebras.
\end{itemize}
Both problems are addressed here exclusively for isometric isomorphisms and (mostly for) automorphic systems. Theorem~\ref{T:iso gives mor} shows that the unitary equivalence of the correspondences $X_{(A, \al)}$ and $X_{(B, \be)}$ is an invariant for isomorphisms between any algebras of the form $\alg(A, \al)$ and $\alg(B, \be)$ considered here. In particular, the tensor algebras of automorphic systems are isometrically isomorphic if and only if the associated correspondences are unitarily equivalent, thus giving an answer to problem (i) in that case. The same answer is also obtained for non-automorphic systems, provided that the ambient $\ca$-algebras are stably finite (Theorem~\ref{T:stably}). The situation is much clearer for automorphic multivariable systems over $\ca$-algebras with trivial center. In that case, Theorem~\ref{trvcent} shows that two semicrossed products (or tensor algebras) are isometrically isomorphic if and only if the corresponding multivariable systems are outer conjugate.  

Any automorphic multivariable system $(A, \al)$ induces a homeomorphic multivariable dynamical system $(\hat{A}, \hat{\al})$ on the Fell spectrum. The latter is a multivariable system of maps acting on a locally compact space and so the concept of piecewise conjugacy from \cite{DK} is meaningful here. In Theorem~\ref{mainthm} we show that the existence of an isometric isomorphism between $\alg(A, \al)$ and $\alg(B, \be)$ implies that the induced dynamical systems  $(\hat{A}, \hat{\al})$ and $(\hat{B}, \hat{\be})$ on the Fell spectra are piecewise conjugate. This gives an answer to problem (ii).

In the final section of the paper, we present two multivariable systems $(A, \al )$ and $(B, \be)$, consisting of a different number $*$-monomorphisms that have isomorphic tensor algebras. Since both $A$ and $B$ equal the Cuntz algebra \cite{Cun} with two generators, this example shows that many of our results (in particular the ones implying $n_\al=n_\be$) do not extend beyond automorphic (actually $*$-epimorphic) systems without making any further assumptions.

In conclusion, we mention the recent number theoretic papers of Cornelissen and Marcolli \cite{C,CM} and their work in graph theory \cite{CM2}. In these papers Cornelissen and Marcolli make essential use of the result of Davidson and Katsoulis~\cite{DK} that piecewise conjugacy is an invariant for isomorphisms between non-selfadjoint operator algebras of classical systems (cf. \cite[Theorem 2]{C}, \cite[Section 6]{CM}, and \cite[Theorem 1.5]{CM2}, the proof of \cite[Theorem 1.6]{CM2}). We hope that the results of this paper will pave the way for further interactions between these important areas of current research.

\section{Preliminaries}\label{S:preliminaries}

A $\ca$-dynamical system $(A, \al)$ consists of a $\ca$-algebra $A$ and a $*$- endomorphism $\al$. (In the sequel all $\ca$-algebras and their $*$-homomorphisms are assumed to be unital.) A multivariable $\ca$-dynamical system  $(A,\al)$ (or simply, a multivariable system) consists of a $\ca$-algebra $A$ and $*$-endomorphisms $\al = (\al_1, \al_2, \dots , \al_{n_{\al}})$ of $A$. If the $\al_1, \al_2, \dots, \al_{n_{\al}}$ happen to be automorphisms, then $(A,\al)$ is said to be an automorphic multivariable system.

We denote by $\T^+(A,\al)$ the tensor algebra of the $\ca$-correspondence $\oplus_{i=1}^{n_{\al}} A_{\al_i}$. (The tensor algebras for $\ca$-correspondences were introduced in \cite{MS, MS2}. The correspondence $\oplus_{i=1}^{n_{\al}} A_{\al_i}$ has been studied in \cite{DK, KK}.)
There is also a related operator algebra, the semicrossed product $A\times_{\al} \bbF_{n_{\al}}^{+}$ associated with $(A,\al)$, where $ \bbF_{n_{\al}}^{+}$ denotes the free semigroup with $n_{\al}$ generators. This is the universal operator algebra generated by a copy of $A$ and contractions $s_1, s_2, \dots s_{n_{\al}}$ satisfying $as_i=s_i \al_i(a)$, $a \in A$, $i=1, 2, \dots, n_{\al}$.

The algebras $\T^+(A,\al)$ and $A\times_{\al} \bbF_{n_{\al}}^{+}$ are not isomorphic in general \cite[Corollary 3.11]{DK} but they do share some common properties which are listed below. As it turns out, there are other operator algebras satisfying these properties and so we take an axiomatic approach in describing them.

\begin{definition} \label{alg}
Let $(A, \al)$ be a  multivariable $\ca$-dynamical system. An operator algebra $\alg (A, \al)$ is said to be \textit{associated with the multivariable system} $(A, \al)$ if it satisfies the following conditions:
\begin{itemize}
\item[(i)] There exists an idempotent mapping $E_0\colon \alg (A, \al) \rightarrow \alg (A, \al)$ with $E_0 (\alg (A, \al) ) =\alg (A, \al)  \cap \alg (A, \al) ^* \simeq A$.
\item[(ii)] There exist elements $s_1, s_2, \dots, s_{n_{\al}} \in \alg (A, \al) $, which are not right divisors of $0$, and satisfy the \textit{covariance relations} $as_i=s_i \al_i(a)$, for all $a \in A$, $i=1,2, \dots, n_{\al}$.
\item[(iii)] $\alg (A, \al)$ is generated as a Banach space by monomials of the form $s_{i_1}s_{i_2}\dots s_{i_k}a$, where $a \in A$, $k \in \bbN$ and $1\leq i_l \leq n_{\al}$, for all $l=1,2,\dots, i_k$ .
\item[(iv)] For each $ 1\leq i \leq n_{\al}$ there exist a bounded idempotent mapping $F_i\colon  \alg (A, \al) \rightarrow \alg (A, \al)$, $ 1\leq i \leq n_{\al}$, which annihilates all monomials except from the ones of the form $s_i a$, $a \in A$, which are left invariant.
\end{itemize}
\end{definition}

Conditions (ii) and (iii) are immediate for both $\T^+(A,\al)$ and $A \times_{\al} \bbF^{+}_{n}$. The verification of conditions (i) and (iv) depends on an argument involving expectations and the Fejer kernel. This argument is by now routine in the non-selfadjoint literature and we omit it (see for example \cite[Section 3.1]{DK}).

We have not opted for maximum generality in the above definition. (That perhaps should be investigated elsewhere.) Instead, we list the minimum requirement so that our theory reaches beyond the tensor algebras or the semicrossed products and includes certain examples that have already appeared in the literature.

\begin{examples}
Let $(A,\al)$ be a multivariable system consisting of mutually commuting $*$-endomorphisms $\al_1, \al_2, \dots , \al_{n}$. Let  $A \times_{\al} \bbZ^{+}_{n}$ denote the universal operator algebra generated by a copy of $A$ and \textit{commuting} contractions $s_1, s_2, \dots, s_{n}$ satisfying the covariance relations in Definition~\ref{alg}\,(ii). It is routine to verify that $A \times_{\al} \bbZ^{+}_{n}$ satisfies the requirements of Definition~\ref{alg} and therefore $A \times_{\al} \bbZ^{+}_{n}$ is an example of an operator algebra associated with $(A, \al)$. Algebras of this type were studied in \cite{DKM, DunPet10, Pow}. If one further asks that the generators  $s_1, s_2, \dots, s_{n}$ are doubly commuting, then we obtain the Nica-covariant semicrossed product studied in \cite{Ful11}.

Alternatively, one may ask for the universal operator algebra generated by a copy of $A$ and a row contraction $(s_1, s_2, \dots, s_{n}) $ consisting of commuting operators and satisfying the covariance relations in Definition~\ref{alg}\,(ii). Note that in that case, the generators $s_1, s_2, \dots, s_{n}$ are not partial isometries but just contractions. In the case where $A=\bbC$, this is the classical Drury-Arveson space studied in \cite{ArP, Arv, DaP} and elsewhere. Additional  examples can be formed by using as a prototype the operator algebras of~\cite{DRS} related to analytic varieties.
\end{examples}

\section{Invertibility of matrices over $\ca$-algebras}

This section contains a technical result (Theorem~\ref{leftinvert}), which may be of independent interest. It shows that if a right invertible rectangular matrix $[b_{ij}]$ over a $\ca$-algebra $B$ satisfies a natural intertwining condition, then $[b_{ij}]$ is actually square and invertible. Just the fact that the matrix has to be square, will allow us to conclude that multivariable dynamical systems with isomorphic operator algebras have necessarily the same dimension. The proof of Theorem~\ref{leftinvert} is algorithmic in nature and this is used in the proof of Theorem~\ref{mainthm}.

\begin{lemma} \label{L:key}
Let $B$, $C$ be $\ca$-algebras and $\phi$, $\psi$ be representations of $B$ onto $C$. Assume that $C$ has trivial center. If $c\in C$ satisfies
\begin{equation} \label{Eq:key}
\phi(b) c = c \psi(b), \foral b\in B,
\end{equation}
 then either $c$ is invertible or $c=0$.
\end{lemma}
\begin{proof}
By taking adjoints in (\ref{Eq:key}), we have $c^*\phi(b)  =  \psi(b) c^*$, for all $b\in B$. Therefore, $cc^* \in \phi(B)'=C'$. By a similar argument $c^*c\in \psi(B)'=C'$. If $c\neq0$, then both $cc^*$ and $c^* c$ are non-zero scalars. This implies that $c$ is a non-zero multiple of a unitary, hence invertible.
\end{proof}

\begin{lemma} \label{ERO}
Let $B$, $C$ be $\ca$-algebras and $\{\phi_i\}_{i=1}^m$, $\{\psi_j\}_{j=1}^n$ be families of representations of $B$ onto $C$ and let $[c_{ij}] \in M_{m,n}(C)$ which intertwines the representations $\{\phi_i\}_{i=1}^m$ and $\{\psi_j\}_{j=1}^n$, i.e.,
\[
\phi_i (b) c_{ij} = c_{ij} \psi_{j}(b),
\]
for all $i=1, 2, \dots m$, $j=1,2, \dots, n$ and $b \in B$.
\begin{itemize}
\item[(i)] If $F_{\pi}$ is the unitary matrix corresponding to a permutation $\pi \in S_n$,  then $[c_{ij}] F_{\pi}$ intertwines the representations $\{\phi_i\}_{i=1}^m$ and $\{\psi_{\pi{(j)}}\}_{j=1}^n$.
\item[(ii)] If $c_{kk}$ is invertible and $E_{hk}$ is the matrix corresponding to the elementary row operation that adds  the $k$-th row multiplied by $-c_{hk}c_{kk}^{-1}$ to the $h$-th row, then $E_{hk}[c_{ij}] $ intertwines the representations  $\{\phi_i\}_{i=1}^m$ and $\{\psi_j\}_{j=1}^n$.
\end{itemize}
\end{lemma}
\begin{proof}
The proof of (i) is straightforward. For proving (ii), we only need to examine elements on the  $h$-th row of $E_{hk} [c_{ij}] $.
Since $[c_{ij} ]$ intertwines the representations $\{\phi_i\}_{i=1}^m$ and $\{\psi_j\}_{j=1}^n$, we have
\begin{align*}
(c_{hj} - c_{hk}c_{kk}^{-1}c_{kj})\psi_j (b) & =
\phi_h (b) c_{hj} - c_{hk}c_{kk}^{-1}\phi_k (b) c_{kj}\\
& =\phi_h (b)  c_{hj} - c_{hk}\psi_k (b) c_{kk}^{-1}c_{kj}\\
& =
\phi_h (b)  c_{hj} - \phi_h (b) c_{hk}c_{kk}^{-1} c_{kj}\\
& =\phi_h (b) (c_{hj} - c_{hk}c_{kk}^{-1}c_{kj}),
\end{align*}
for all $b \in B$ and $i = 1, 2, \dots , m$, as desired.
\end{proof}

\begin{lemma}[Gaussian Elimination] \label{L:quasisim}
Let $B$, $C$ be $\ca$-algebras and $\{\phi_i\}_{i=1}^m$, $\{\psi_j\}_{j=1}^n$ be families of representations of $B$ onto $C$, with $m \geq n$. Let $[c_{ij}] \in M_{m,n}(C)$ which intertwines the representations $\{\phi_i\}_{i=1}^m$ and $\{\psi_j\}_{j=1}^n$. If $[c_{ij}] $ is right invertible and $C$ has trivial center, then $m=n$ and $[c_{ij}] $ is quasisimilar\,\footnote{Note that in our context, a quasi similarity is implemented by invertible operators.} in $M_n (C)$ to a diagonal invertible matrix.
\end{lemma}
\begin{proof}
We will first produce invertible matrices $E' \in M_m(C)$ and $F' \in M_n( C)$ so that the matrix $E'[c_{ij}] F'$ has invertible diagonal entries and all entries below the diagonal equal to $0$. Once this is done, $m = n$, or otherwise $E'[c_{ij}] F'$ would have a zero row, a contradiction to the right invertibility of $[c_{ij}] $. We do this by using a variant of the Gaussian elimination on $[c_{ij}] $.

Start with the first column. Since $[c_{ij}] $ is right invertible, there exists at least one entry on the first row, say $c_{1j_{1}}$, which is non-zero. By Lemma~\ref{L:key}, $c_{1j_{1}}$ is invertible. Let $F_{(1\, j_{1})}$ be as in Lemma~\ref{ERO}, where $(1\, j_{1})$ is the transposition between $1$ and $j_1$. If the $(i,1)$-entry of $[c_{ij}] F_{(1\, j_{1})}$ is not zero, then let $E_{i1}$ be as in Lemma~\ref{L:key}\,(ii), but for the matrix  $[c_{ij}] F_{(1\, j_{1})}$. Otherwise, set $E_{i1}=I$. Then the matrix
\begin{equation} \label{quasisim}
\left( \prod_{i=1}^{m} \, E_{i1} \right) \left[ c_{ij} \right] F_{(1\, j_{1})}
\end{equation}
has its $(1,1)$-entry invertible and all entries below the $(1, 1)$-entry equal to $0$. Furthermore (\ref{quasisim}) is right invertible and by Lemma~\ref{ERO}, it intertwines the representations $\{\phi_i\}_{i=1}^m$ and $\{\psi_{\pi{(j)}}\}_{j=1}^n$, where $\pi=(1 \, j_1)$. Hence we can continue the Gaussian elimination with the second column of (\ref{quasisim}) this time. One of the entries on the second row, say the $(2, j_2)$-entry will be non-zero, and hence by Lemma~\ref{L:key} invertible. Multiply (\ref{quasisim}) from the right by $F_{(2 \, j_{2})}$ and from the left by invertible matrices $ E_{i2} $, coming from Lemma~\ref{ERO}, in order to zero all entries on the second column which are below the diagonal. Continuing in this fashion, we eventually produce the desired upper triangular matrix $E'[c_{ij}] F'$.

Since the diagonal entries of $E'[c_{ij}] F'$ are invertible, an elementary application of the Gaussian elimination produces an invertible matrix $E''$ so that $E''E'[c_{ij}] F'$ is diagonal and the conclusion follows.
\end{proof}

\begin{theorem} \label{leftinvert}
Let $B$ be a $\ca$-algebra and $\{ \be_i \}_{i=1}^{m}$, $\{ \be'_j \}_{j=1}^{n}$  be families of $*$-epimorphisms of $B$, with $m \geq n$. Let $[b_{ij}] \in M_{m,n}(B)$ which intertwines $\{ \be_i \}_{i=1}^{m}$ and $\{ \be'_j \}_{j=1}^{n}$.
\begin{itemize}
\item[(i)] If $[b_{ij}]$ is right invertible, then $m=n$ and $[b_{ij}]$ is invertible in $M_n (B)$.
\item[(ii)] If $[b_{ij}]$ is right invertible and $B$ has trivial center, then $[b_{ij}]$ is quasisimilar to a diagonal matrix which intertwines $\{ \be_i \}_{i=1}^{n}$ and \break $\{ \be'_{\pi(i)} \}_{i=1}^{n}$, for some permutation $\pi \in S_n$.
\end{itemize}
\end{theorem}
\begin{proof}
We have already proved (ii). For (i), let $[d_{ij}] \in M_{m,n}(B)$ be the right inverse of $[b_{ij}]$. If $\rho$ is any irreducible representation of $B$, then  $[\rho(b_{ij})]$ intertwines the representations $\{ \rho \be_i \}_{i=1}^{m}$ and $\{ \rho \be'_j \}_{j=1}^{n}$. Hence Lemma~\ref{L:quasisim} implies that $m=n$ and also that $\rho^{(n)}([d_{ij}]) \equiv [\rho(d_{ij})]$ is the inverse for $[\rho(b_{ij})]$.

Let $( \rho_s )_s$ be a family of irreducible representations of $B$ that separates the points. By \cite[Theorem 6.5.1]{Murphy}, $\id_n \otimes (\oplus_s \rho_s)$ is a faithful representation of $M_n(\bbC) \otimes B$ and so
\[
\bigoplus_s \rho_s^{(n)} \simeq \bigoplus_s (\id_n \otimes \rho_s)
\]
is a faithful representation for $M_n (B)$. In that representation, the previous paragraph shows that $[d_{ij}]$ is the inverse of $[b_{ij}]$.
\end{proof}

\section{The main results}

Let $\alg(A,\al)$ and $\alg (B, \be)$ be operator algebras associated with the multivariable systems $(A,\al)$ and $(B, \be)$ respectively, and let $\ga\colon \alg(A,\al) \rightarrow \alg (B, \be)$ be an isometric isomorphism. Since $\ga$ is isometric, a similar argument as in \cite[Proposition~2]{DKsimple} implies that $\ga|_A$ is a $*$-monomorphism that maps $A$ onto $B$. We will be denoting $\ga|_A$ by $\ga$ as well.

Let $s_1 , s_2, \dots, s_{ n_{\al}}$ and $t_1, t_2, ,\dots, t_{n_{\be}}$ be the generators in $\alg (A,\al)$ and $\alg (B,\be)$ respectively, and let $b_{ij} \equiv F_i(s_j)$ so that
\[
\ga(s_j)= b_{0j} + t_1 b_{1j} + t_2 b_{2j} + \cdots + t_{n_{\be}} b_{ n_{\be} j} + Y
\]
with $E_0(Y)=F_1(Y)= \cdots = F_{n_\be}(Y)=0$.
Since $\ga$ is a homomorphism,
\begin{align*}
\ga(a)\ga(s_j)=\ga(as_j)=\ga(s_j \al_j(a))=\ga(s_j) \ga\al(a),
\end{align*}
for all $a\in A$. Hence, $\be_i\ga(a)b_{ij}=b_{ij}\ga\al_j(a)$,
$a \in A$, and so
\begin{equation} \label{eq:conj}
\be_i(b) b_{ij}= b_{ij} \ga\al_j\ga^{-1}(b),
\end{equation}
for all $b \in B$. Therefore, the \textit{matrix $[b_{ij}]$ associated with the isomorphism $\ga$} intertwines $\{ \be_{i} \}_{i=1}^{n_{\be}}$ and $\{\ga \al_{j} \ga^{-1}\}_{j=1}^{n_{\al}}$.

\begin{lemma}\label{L:ti}
Let $\alg(A,\al)$ be an operator algebra associated with the multivariable system $(A,\al)$. If a sequence $(a_k)_k$ in $A$ satisfies  $\lim_k s_i a_k =0$, for some $1\leq i \leq n_{\al}$, then $\lim_k a_k=0$.
\end{lemma}
\begin{proof}
Since $s_i$ is not a right divisor of zero, the $\ca$-algebra $A$ equipped with the seminorm $\nor{a}_i \equiv \nor{s_ia}$, $a\in A$, becomes a Banach space. The identity map $\id\colon (A,\nor{\cdot}) \rightarrow (A,\nor{\cdot}_i)$ is continuous, hence by the open mapping theorem bicontinuous, and the conclusion follows.
\end{proof}

\begin{lemma}\label{L:Fer}
Let $\alg(A,\al)$ and $\alg (B, \be)$ be operator algebras associated with the multivariable systems $(A,\al)$ and $(B, \be)$ respectively, and let $\ga: \alg(A,\al) \rightarrow \alg (B, \be)$ be an isometric isomorphism. Then, for a given tuple $(y_1, y_2, \dots, y_{n_{\be}}) \in \oplus_{\, 1}^{n_{\be}}B$ there exist a sequence $ \left( (x_{1}^{k} , x_{2}^{k}, \dots, x_{n_{\al}}^{k}) \right)_k$ in $\oplus_{\, 1}^{n_{\al}}B$ such that
\begin{align*}
y_i & = \lim_k b_{i1}x^k_1 + b_{i2} x^k_2 + \cdots + b_{i n_{\al}} x^k_{n_{\al}}, \quad i=1,2,,\dots,n_{\be},
\end{align*}
where $[b_{ij}]$ is the matrix associated with $\ga$.
\end{lemma}
\begin{proof}
Let $\nu=\nu_1\dots \nu_q$ be a word on $\emptyset, 1, \dots, n_{\al}$ and let $s_{\nu}=s_{\nu_1}s_{\nu_2}\dots s_{\nu_q}$,  with the understanding that $s_\emptyset$ denotes an element in $A$.  Then
\begin{align*}
F_i(\gamma(s_\nu))
&= F_i(\gamma(s_{\nu_1}))E_0(\gamma(s_{\nu_2}))\cdots E_0(\gamma(s_{\nu_q})) + \dots \\
&\hspace{1.3in} \dots + E_0(\gamma(s_{\nu_1}))\cdots E_0(\gamma(s_{\nu_{q-1}}))F_i(\gamma(s_{\nu_q})) \\
& =  F_i(\gamma(s_{\nu_1}))E_0(\gamma(s_{\nu_2})) \cdots E_0(\gamma(s_{\nu_q})) + \dots \\
&\hspace{1.3in} \dots + E_0(\gamma(s_{\nu_1})) \cdots E_0(\gamma(s_{\nu_{q-1}})) F_i(\gamma(s_{\nu_q})) \\
& \hspace{-.35in} = \sum_r E_0(\gamma(s_{\nu_1})) \cdots E_0(\gamma(s_{\nu_{r-1}})) F_i(\gamma(s_{\nu_r})) E_0(\gamma(s_{\nu_{r+1}})) \cdots E_0(\gamma(s_{\nu_q})).
\end{align*}
By equation (\ref{eq:conj}) and for suitable $y_r, y'_r \in B$ we obtain,
\begin{align*}
F_i(\gamma(s_\nu))
& = \sum_r y_r  t_ib_{i \nu_r}  y'_r
  = \sum_r t_i \be_i(y_r)b_{i \nu_r}  y'_r  \\
& = \sum_r t_i b_{i \nu_r}  \ga\al_{\nu_r} \ga^{-1} ( y_r )  y'_r
  = \sum_{r=1}^{n_{\al}} t_i b_{i r}  x_r.
\end{align*}
Therefore
\[
F_i(\gamma(s_\nu)) = t_i(b_{i 1} x_1 + b_{i2}x_2 + \dots + b_{i{n_{\al}}} x_{n_{\al}}).
\]
The same follows for linear combinations of the monomials $\gamma(s_\nu)$. For example, for two words $\mu=\mu_1 \cdots \mu_w$ and $\nu=\nu_1\cdots \nu_q$ we obtain that
\begin{align*}
F_i\gamma(s_\mu + s_\nu) & = F_i\gamma(s_\mu) + F_i\gamma(s_\mu) \\
& = t_i(b_{i\nu_1} x_{\nu_1} + b_{i \nu_2} x_{\nu_2} + \cdots + b_{i \nu_q} x_{\nu_q}) +\\
& \hspace{1cm} + t_i(b_{i\mu_1} x_{\mu_1} + b_{i \mu_2} x_{\mu_2} + \cdots + b_{i \mu_q} x_{\mu_q})\\
& = t_i(b_{i 1} x_1 + b_{i2}x_2 + \dots + b_{i{n_{\al}}} x_{n_{\al}}) + \\
& \hspace{1cm} +  t_i(b_{i 1} x_1' + b_{i2}x_2' + \dots + b_{i{n_{\al}}} x_{n_{\al}}')\\
& = t_i( b_{i 1} x_1'' + b_{i2}x_2'' + \dots + b_{i{n_{\al}}} x_{n_{\al}}''),
\end{align*}
by introducing the appropriate zeros. Since $t_1y_1 +t_2y_2 + \dots +t_{n_{\be}}y_{n_{\be}}$ is a limit of such linear combinations, we obtain sequences
$\left((x_{1}^{k} , x_{2}^{k}, \dots, x_{n_{\al}}^{k}) \right)_k$
such that
\begin{align*}
t_iy_i & = \lim_k t_i(b_{i1}x^k_1 + b_{i2} x^k_2 + \cdots + b_{i n_{\al}} x^k_{n_{\al}})
\end{align*}
for all $1\leq i \leq n_{\be}$, and Lemma \ref{L:ti} finishes the proof.
\end{proof}

\begin{proposition}\label{P:right inv}
Let $\alg(A,\al)$ and $\alg (B, \be)$ be operator algebras associated with the multivariable systems $(A,\al)$ and $(B, \be)$ respectively, and let $\ga\colon \alg(A,\al) \rightarrow \alg (B, \be)$ be an isometric isomorphism. Then the matrix $[b_{ij}]$ associated with $\ga$ is right invertible.
\end{proposition}
\begin{proof}
By Lemma \ref{L:Fer}, for any tuple $(y_1,\dots,y_{n_{\be}})$ there exists a sequence $\left( (x_{1}^{k} , x_{2}^{k}, \dots, x_{n_{\al}}^{k}) \right)_k$ such that
\begin{align*}
y_i & = \lim_k b_{i1}x^k_1 + b_{i2} x^k_2 + \cdots + b_{i n_{\al}} x^k_{n_{\al}}
\end{align*}
for all $1\leq i \leq n_{\be}$.
Hence for the tuple $(1,0,\dots,0)$ and for $\eps< \frac{1}{n_{\al}n_{\be}}$ there are $x_{j1}$, for $1\leq j \leq n_\al$, such that
\begin{align*}
& \nor{\de_{i1} - b_{i1}x_{11} + b_{i2} x_{21} + \cdots + b_{in_{\al}} x_{n_{\al}1}} < \eps,
\end{align*}
for all $i=1,\dots,n_{\be}$. Repeating for $(0,1,\dots,0), \dots, (0,0\dots,1)$, we obtain elements $x_{ij}$, for $1\leq i \leq n_{\al}$ and $1 \leq j \leq n_{\be}$, such that
\begin{align*}
\nor{\de_{ij} - \sum_{k=1}^{n_{\al}} b_{ik}x_{kj} } <\eps < \frac{1}{n_{\al}n_{\be}}.
\end{align*}
Hence,
\begin{align*}
\nor{I_n - [b_{ij}][x_{ij}]}
& = \nor{[\de_{ij} - \sum_{k=1}^{n_{\al}} b_{ik}x_{kj}]}\\
& \leq \sum_{i,j}\nor{\de_{ij} - \sum_{k=1}^{n_{\al}} b_{ik}x_{kj}}
 < \sum_{i,j} \frac{1}{n_{\al}n_{\be}} = 1.
\end{align*}
Therefore $[b_{ij}][x_{ij}]$ is invertible, hence $[b_{ij}]$ is right invertible.
\end{proof}

From Proposition~\ref{P:right inv} and Theorem~\ref{leftinvert}, we obtain the following key result.

\begin{theorem}\label{T:right inv}
Let $\alg(A,\al)$ and $\alg (B, \be)$ be operator algebras associated with the automorphic multivariable systems $(A,\al)$ and $(B, \be)$ respectively, and let $\ga\colon \alg(A,\al) \rightarrow \alg (B, \be)$ be an isometric isomorphism. Then $n_{\al}= n_{\be}$ and the matrix $[b_{ij}]$ associated with $\ga$ is invertible in $M_n (B)$.
\end{theorem}

As a first application of Theorem~\ref{T:right inv}, we obtain that the unitary equivalence class of $X_{(A, \al)}$ is an isomorphism invariant for $\alg(A, \al)$, which is complete in the case of tensor algebras. Note that the unitary equivalence class of $X_{(A, \al)}$ is easy to describe here: $X_{(A,\al)}$ and $X_{(B,\be)}$ are \emph{unitarily equivalent} if and only if there is a $*$-isomorphism $\ga\colon A \rightarrow B$ and a unitary matrix $[u_{ij}] \in M_{n_\be,n_\al}(B)$ that intertwines $\{\be_i\}_{i=1}^{n_\be}$ and $\{\ga\al_j\ga^{-1}\}_{j=1}^{n_\al}$.
(When $n_\al=n_\be$ and $[u_{ij}]$ happens to be diagonal up to a permutation, then the multivariable systems $(A,\al)$ and $(B,\be)$ are said to be \emph{outer conjugate}.)

\begin{theorem}\label{T:iso gives mor}
Let $(A,\al)$ and $(B,\be)$ be two automorphic multivariable $\ca$-dynamical systems.
\begin{enumerate}
\item If $\alg(A,\al)$ and  $\alg (B, \be)$ are isometrically isomorphic then $n_\al=n_\be$ and the correspondences $X_{(A,\al)}$ and $X_{(B,\be)}$ are unitarily equivalent.
\item $\T^+(A,\al)$ and $\T^+(B,\be)$ are isometrically isomorphic if and only if the correspondences $X_{(A,\al)}$ and $X_{(B,\be)}$ are unitarily equivalent.
\end{enumerate}
\end{theorem}
\begin{proof}
Let $[b_{ij}]$ be the matrix associated with an isometric isomorphism $\ga\colon \alg(A,\al) \rightarrow \alg(B,\be)$. By Theorem~\ref{T:right inv}, $[b_{ij}]$ is invertible. If $[b_{ij}] = w \left| [b_{ij}] \right| $ is the polar decomposition of $[b_{ij}]$, then the unitary $w$ intertwines $\{ \be_{i} \}_{i=1}^{n_{\be}}$ and $\{\ga \al_{j}\ga^{-1} \}_{j=1}^{n_{\al}}$. Thus the pair $(\ga,w)$ induces the desired unitary equivalence.

To end the proof, recall that when $X_{(A,\al)}$ and $X_{(B,\be)}$ (resp. $(A,\al)$ and $(B,\be)$) are unitarily equivalent (resp. outer conjugate) then the tensor algebras (resp. the semicrossed products) are completely isometrically isomorphic.
\end{proof}

\begin{corollary} \label{semipliestens}
Let $(A,\al)$ and $(B,\be)$ be two automorphic multivariable $\ca$-dynamical systems. If the semicrossed products $A\times_{\al} \bbF_{n_{\al}}^{+}$  and $B\times_{\be} \bbF_{n_{\be}}^{+}$ \textup{(}or some $\alg(A,\al)$ and $\alg(B,\be)$\textup{)} are isometrically isomorphic, then the tensor algebras $\T^+(A,\al)$ and $\T^+(B,\be)$ are also isometrically isomorphic.
\end{corollary}

The converse of Corollary~\ref{semipliestens} does not hold. This follows from \cite[Example 3.24]{DK}.

In the case where the multivariable system acts on a $\ca$-algebra with a trivial center, we obtain that outer conjugacy is a complete invariant for isomorphisms between semicrossed products, by combining Theorem \ref{T:iso gives mor} with Lemma \ref{L:quasisim}.

\begin{theorem} \label{trvcent}
Let $(A,\al)$ and $(B,\be)$ be two automorphic multivariable $\ca$-dynamical systems and assume that $A$ has trivial center. Then the following are equivalent:
\begin{enumerate}
\item $A\times_{\al} \bbF_{n_{\al}}^{+}$  and $B\times_{\be} \bbF_{n_{\be}}^{+}$ are isometrically isomorphic.
\item $\T^+(A,\al)$ and $\T^+(B,\be)$ are isometrically isomorphic.
\item $X_{(A,\al)}$ and $X_{(B,\be)}$ are unitarily equivalent.
\item $(A,\al)$ and $(B,\be)$ are outer conjugate.
\end{enumerate}
Furthermore, if there exist $\alg(A,\al)$ and $\alg(B,\be)$ which are isometrically isomorphic, then any of the above conditions holds, and $n_\al=n_\be$.
\end{theorem}

Let $A$ be a unital $\ca$-algebra and let $P(A)$ be its pure state space equipped with the $w^*$-topology. The \textit{Fell spectrum} $\hat{A}$ of $A$  is the space of unitary equivalence classes of non-zero irreducible representations of $A$. (The usual unitary equivalence of representations will be denoted as $\sim$.) The GNS construction provides a surjection $P(A) \rightarrow \hat{A} $ and $\hat{A}$ is given the quotient topology. There is another more convenient description of the (Fell) spectrum of $A$ due to Ernest \cite{Ernest}. Let $\H_A$ be a fixed Hilbert space of dimension equal the cardinal $\kappa$ of a dense subset of $A$. A \textit{railway representation} $\rho$ of $A$ is a representation  which is unitarily equivalent to the $\kappa$-ampliation of some irreducible representation $\rho_0$ of $A$. In particular, if $\rho$ is a railway representation, then $\rho(A)$ has trivial center. Let $\R(A)$ denote the space of all railway representations of A acting on $ \H_A$, equipped  with the topology of pointwise convergence relative to the strong operator topology. Ernest shows in \cite{Ernest} that the canonical surjection $\R(A) \rightarrow \hat{A}$, which associates with each railway representation $\rho \in \R (A)$ the unitary equivalence class of the irreducible representation $\rho_0$ associated with $\rho$, is both open and continuous. Hence, the space $ \R(A)\slash \sim$ equipped with the quotient topology is homeomorphic in a canonical way with the spectrum of $A$. (Indeed two railway representations are unitarily equivalent if and only if their associated irreducible representations are unitarily equivalent.) For the sequel, we adopt Ernest's picture for the spectrum, i.e., $\hat{A}=  \R(A)\slash \sim$.
We require the following elementary fact regarding the open sets in $\hat{A}$.

\begin{proposition} \label{open}
Let $A$ be a $\ca$-algebra, $\rho \in \R(A)$,  $a \in A$ and $\xi \in \H_A$. Then the set
\[
U(\rho \mid a, \xi, \epsilon) \equiv \{[\rho'] \mid \rho' \in \R(A), \, \|(\rho -\rho')(a)\xi \|< \epsilon \}
\]
is open in $\hat{A}$.
\end{proposition}
\begin{proof}
It is enough to show that the set
\[
U'= \{ \ad_{w} \rho' \mid w \in B(\H_A) \mbox{ unitary, }  \|(\rho -\rho')(a)\xi \|< \epsilon \}
\]
is open with respect to the topology of pointwise convergence in $\R(A)$.

Let $\ad_w \rho' \in U' $ arbitrary and let $\delta >0$ so that
\begin{equation} \label {delta}
\|(\ad_w \rho' -\ad_w \rho)(a)w\xi\|=\|(\rho-\rho')(a)\xi\|=\epsilon - \delta
\end{equation}
Consider now the open set
\[
O_{\rho' , w} =\{\rho'' \in \R(A) \mid \| (\ad_w \rho'-\rho'')(a)w\xi\| <\delta\}.
\]
If we show that $O_{\rho' , w} \subseteq U'$, then $U'$ will be the union of open sets and the conclusion will follow.
Towards this end, let $\rho'' \in O_{\rho' , w} $. By (\ref{delta}), $\|(\ad_w\rho-\rho'')(a)w\xi\| <\epsilon$ and thus
\begin{align*}
\|(\ad_{w^*} \rho'' -\rho)(a)\xi \|
& = \|w (\ad_{w^*}\rho''- \rho) \xi\|  =\|(\rho''- \ad_w \rho )(a)w\xi\| < \epsilon.
\end{align*}
Hence $\ad_{w^*} \rho''  \in U'$ and so $ \rho'' \in U'$, as desired.
\end{proof}

Let $X$ and $Y$ be topological spaces and let $\sigma= (\sigma_1, \sigma_2, \dots, \sigma_n)$ and $\tau=(\tau_1, \tau_2, \dots , \tau_n)$ be multivariable dynamical systems consisting of selfmaps of $X$ and $Y$ respectively. Davidson and Katsoulis \cite[Definition 3.16]{DK} define $(X, \sigma)$ and $(Y, \tau)$ to be \textit{piecewise conjugate} if there exists a homeomorphism $\phi\colon X\rightarrow Y$ and an open cover $\{ U_g \mid g \in S_n \}$ of $Y$ so that
\[
\tau_i =\phi \, \sigma_{g(i)} \, \phi^{-1}, \text{ for each $g \in S_n$ and $1\leq i\leq n$.}
\]

If $A$ is a $\ca$-algebra, then any automorphism (resp. multivariable system) $\al$ of $A$ induces a homeomorphism (resp. multivariable dynamical system) $\hat{\al}$ on its Fell spectrum $\hat{A}$, that maps the equivalence class $[\rho]$ of a railway representation to $[\rho\al]$.

\begin{theorem} \label{mainthm}
Let $(A,\al)$ and $(B,\be)$ be automorphic multivariable $\ca$-dynamical systems and assume that there exist associated operator algebras $\alg(A,\al)$ and $\alg (B, \be)$ which are isometrically isomorphic. Then the multivariable systems $(\hat{A},\hat{\al})$ and $(\hat{B},\hat{\be})$  are piecewise conjugate.
\end{theorem}
\begin{proof}
Let $\ga\colon \alg(A,\al) \rightarrow \alg (B, \be)$ be an isometric isomorphism. We will show that the mapping $\hat{\ga}^{-1}\colon \hat{A} \rightarrow \hat{B}$ is the homeomorphism implementing the desired piecewise conjugacy between $(\hat{A}, \hat{\al})$ and $(\hat{B}, \hat{\be})$. In order to prove that we will verify that around every point in $\hat{B}$, there is an open set so that the maps $\hat{ \be}_i$, $i=1, 2, \dots, n_{\al}$,  and $\hat{\ga}^{-1}\hat{\al}_{j}\hat{\ga}= \! \! \! \widehat{\phantom{w} \ga \al_j\ga^{-1}}$,  $j=1, 2, \dots, n_{\be}$, when restricted there, they are conjugate.

Let $\rho \in \R(B)$ and pick one of the $\hat{ \be}_1, \hat{ \be}_2, \dots, \hat{ \be}_{n_{\be}}$, say $\hat{ \be}_1$. Let
\begin{equation*}
\{\hat{ \be}_1, \hat{ \be}_2, \dots, \hat{ \be}_r,   \! \! \! \widehat{\phantom{w} \ga \al_1\ga^{-1}},   \! \! \!  \widehat{\phantom{w} \ga \al_2\ga^{-1}}, \dots,  \! \! \! \widehat{\phantom{w} \ga \al_d\ga^{-1}} \}
\end{equation*}
be the mappings from the collection
\begin{equation} \label{germ}
\{\hat{ \be}_1, \hat{ \be}_2, \dots, \hat{ \be}_{n_{\al}},   \! \! \!  \widehat{\phantom{w} \ga \al_1\ga^{-1}},   \! \! \!  \widehat{\phantom{w} \ga \al_2\ga^{-1}}, \dots,  \! \! \!  \widehat{\phantom{w} \ga \al_{n_{\be}}\ga^{-1}} \}
\end{equation}
that ``eventually'' agree with $\hat{\be}_1$ around $[\rho]$ (the germ of $\hat{\be}_1$). By that we mean that there is an open set  $U \subseteq \hat{B}$ containing $[\rho]$ so that
\[
\widehat{\phantom{pp} \ga \al_j\ga^{-1}}|_{U}=\hat{ \be_i}|_U,
\]
for all $i=1,2,\dots, r$ and $j=1,2,\dots,d$. Furthermore, given any open set $ U' \subseteq U\subseteq  \hat{B}$ containing $[\rho]$ and  any $i \leq r$, $j>d$, we have
\begin{equation} \label{contradict}
\widehat{\phantom{pp} \ga \al_j\ga^{-1}}|_{U'} \neq \hat{\be_1}|_{U'} = \hat{\be_i}|_{U'}.
\end{equation}

We are to show that $r=d$.  Once this has been established, an easy partitioning argument for the collection (\ref{germ}) into germs finishes the proof of the theorem.

By way of contradiction we assume that $r>d$, or otherwise we exchange the roles of $A$ and $B$ and their corresponding automorphisms. (We do not exclude the possibility that $d=0$.)

\vspace{.05in}
\noindent
\textit{Claim.}\,If $[b_{ij}]$ is the matrix associated with $\ga$, then $\rho(b_{ij})=0$,
for all $i=1,2,\dots,r$ and $j>d$.

\noindent \textit{Proof of the Claim.}
Indeed, let $\epsilon >0$ and $\xi \in \H_A$ and let
\[
U(\rho \mid b_{ij}, \xi, \epsilon) = \{ [ \rho' ] \in \hat{B} \mid \| (\rho-\rho')(b_{ij})\xi\| < \epsilon\}.
\]
which is open by Proposition~\ref{open}.
Hence $U(\rho \mid b_{ij}, \xi, \epsilon)\bigcap U$ is also open and so by (\ref{contradict}) implies the existence of $[\rho' ] \in U(\rho \mid b_{ij}, \xi, \epsilon) \bigcap U$ so that
\[
\widehat{\phantom{pp} \ga \al_j\ga^{-1}}(\rho') \neq \hat{\be_1}(\rho') = \hat{\be_i}(\rho').
\]
Hence $\rho' \ga \al_{j} \ga^{-1} \nsim \rho \be_i$ and so Lemma \ref{L:key} implies that $\rho'(b_{ij})=0$. Therefore
\[
\| \rho(b_{ij})\xi\| =  \| (\rho-\rho')(b_{ij})\xi\| < \epsilon,
\]
for all $\xi \in \H_A$ and $\epsilon >0$, which proves the claim.

\vspace{.05in}
By Theorem~\ref{T:right inv} the matrix $[\rho(b_{ij})]$ is invertible. It also intertwines the representations $\{\rho \be_{i} \}_{i=1}^{n}$ and $\{\rho \ga \al_{j} \ga^{-1}\}_{j=1}^{n}$ and so we can perform Gaussian elimination as in Lemma~\ref{L:quasisim}. However, the claim above implies that when we reach at the $d+1$ row, there will be no non-zero element on that particular row. This contradicts the invertibility of $[\rho(b_{ij})]$.
\end{proof}

The Fell topology on the spectrum of a simple $\ca$-algebra is the discrete topology. Therefore for such $\ca$-algebras Theorem~\ref{mainthm} says nothing more than the invariance of dimension, i.e., $n_{\al}=n_{\be}$. It turns out that an appropriate modification of the Fell spectrum, combined with the techniques of Theorem~\ref{mainthm} yields a finer invariant that can handle simple $\ca$-algebras and actually a bit more.

\begin{definition}
Let $A$ be a $\ca$-algebra and $\rho , \rho'$ be representations of $A$. We say that $\rho$ and $\rho'$ are \textit{strongly equivalent} (denoted $\rho \ssim \rho'$) if $\rho(A) = \rho'(A)$ and there exists a unitary operator $w \in \rho (A)$ so that $\rho'=\ad_w \rho$.
\end{definition}

Let $\S(A)$ be the space of all non-degenerate representations of $A$ on $\H_A$ with trivial center.
The equivalence $\ssim$ partitions $\S(A)$ into equivalence classes and the collection of all these classes will be denoted as $\tilde{A}$. We equip $\tilde{A}$ with the smallest topology so that sets of the form
\[
 \{[\rho'] \mid \rho' \in \S(A), \, \|(\rho -\rho')(a)\xi \|< \epsilon \},
\]
where $\rho \in \S(A)$, $a \in A$,  $\xi \in \H_{A}$, $\epsilon >0$, are open. The topologized space $\tilde{A}$ gives a finer notion of spectrum than that of the Fell spectrum and it coincides with the Gelfand spectrum in the commutative case. It is easy to see that any automorphism (resp. multivariable system) $\al$ of $A$ induces a homeomorphism (resp. multivariable dynamical system) $\tilde{\al}$ on $\tilde{A}$, that maps the equivalence class $[\rho]$, $\rho \in \S(A)$ onto $[\rho\al]$.
A verbatim repetition of the proof of Theorem~\ref{mainthm} (with hats replaced by tildes) yields the following.

\begin{theorem} \label{secondmainthm}
Let $(A,\al)$ and $(B,\be)$ be automorphic multivariable $\ca$-dynamical systems and assume that there exist associated operator algebras $\alg(A,\al)$ and  $\alg (B, \be)$ which are isometrically isomorphic. Then the multivariable systems $(\tilde{A},\tilde{\al})$ and $(\tilde{B},\tilde{\be})$  are piecewise conjugate.
\end{theorem}

In certain cases Theorems~\ref{mainthm} and \ref{secondmainthm} provide a complete invariant for isometric isomorphism between tensor algebras, e.g., multivariable dynamics with two generators over commutative $\ca$-algebras \cite{DK}. Nevertheless, none of these results provides a complete invariant in general. This follows from the work of Davidson and Kakariadis \cite{DKak} and an example of Kadison and Ringrose \cite{KadR}.

\begin{example}
In \cite{KadR} Kadison and Ringrose show that there exists a (homogeneous) $\ca$-algebra $A$ and an automorphism $\al$ of  $A$ which is universally weakly inner but not inner. If the converse of Theorem~\ref{mainthm} were valid for tensor algebras, then $A \times_{\al} \bbZ^+$ and $A\times_\id \bbZ^+$ would be isomorphic and hence outer conjugate by \cite{DKak}. But this would imply that $\al$ is inner, a contradiction.
\end{example}

\section{Concluding Remarks and open problems}

One of the consequences of our theory is the \textit{invariance of dimension}:  if  $(A,\al)$ and $(B, \be)$ are multivariable systems consisting of $*$-automorphisms and $\alg(A,\al)$ and $\alg (B, \be)$ are isometrically isomorphic as operator algebras, then $n_{\al}= n_{\be}$. Furthermore the invariance of dimension is implicit in both the statements of piecewise conjugacy and outer conjugacy and therefore it is a corollary of both Theorem~\ref{mainthm} and Theorem~\ref{trvcent}. The following example shows that the invariance of dimension does not hold for arbitrary multivariable systems.

\begin{example}\label{E:counter}
Let $A=B=\O_2$, $\al=(\al_1, \al_2)$, with $\al_1=\al_2=\id$, and let
\[
\be(x) = S_1xS_1^* + S_2xS_2^*, \, \, x \in \O_2,
\]
where $S_1, S_2$ are the canonical generators of $\O_2$. Then the tensor algebras $\T^{+} (A, \al)$ and $\T^{+} (B ,\be) $ are (completely) isometrically isomorphic.

Indeed, $\be(x)S_1 = S_1x$ and $\be(x)S_2 = S_2 x$, for all $x\in \O_2$. Hence, the unitary matrix $U=\begin{bmatrix} S_1 & S_2 \end{bmatrix}$ intertwines $\{\be\}$ and $\{\al_1,\al_2\}$ and so $X_{(A, \al)}$ and $X_{(B, \be)}$ are unitarily equivalent.
\end{example}

The above example does not exclude the possibility that the $\ca$- correspondence is an isomorphism invariant for arbitrary multivariable systems.

\begin{question}
Let $(A,\al)$ and $(B,\be)$ be multivariable dynamical systems consisting of arbitrary $*$-endomorphisms. Assume that there exist associated operator algebras $\alg(A,\al)$ and  $\alg (B, \be)$ which are isometrically isomorphic. Does it follow that the $\ca$-correspondences $X_{(A,\al)}$ and $X_{(B,\be)}$ are unitarily equivalent?
\end{question}

Question~1 has a positive answer when both multivariable systems consists of $*$-epimorphisms. Indeed, Theorem~\ref{T:right inv} and its consequences are valid also in this case, with the same proofs. Because of Example~\ref{E:counter}, we emphasize the assumption that both families consist of $*$-epimorphisms.
Question 1 also has a positive answer in the following case.

\begin{theorem}\label{T:stably}
Items $(1)$ and $(2)$ of Theorem~\ref{T:iso gives mor} hold for arbitrary classical systems. 

The same holds for multivariable systems of stably finite $\ca$-algebras with arbitrary $*$-endomorphisms, provided that $ n_{\al} = n_{\be}$.
\end{theorem}
\begin{proof}
Suppose that $[b_{ij}]$ is the matrix associated with the isomorphism $\ga \colon \alg(A,\al) \rightarrow \alg(B,\be)$. By Proposition~\ref{P:right inv} $[b_{ij}]$ is a right-invertible rectangular matrix. If $B$ is commutative then $n_\be \leq n_\al$ and by symmetry on $\ga^{-1}$ we obtain that $n_\be = n_\al$. The finiteness condition for $B$ implies that $[b_{ij}]$ is invertible.
\end{proof}

When the $\ca$-algebras are commutative we obtain that unitary equivalence implies piecewise conjugacy of the systems, by passing though the isomorphism of the tensor algebras \cite[Theorem 3.22]{DK}. In certain cases, piecewise conjugacy implies also isometric isomorphism of the tensor algebras \cite[Theorem 3.25]{DK}, thus in these cases piecewise conjugacy and unitary equivalence of the $\ca$-correspondences coincide.

\begin{question}
Does piecewise conjugacy imply unitary equivalence of the $\ca$-correspondences for classical dynamical systems in general?
\end{question}

We are also interested in piecewise conjugacy over the Jacobson spectra.

\begin{question}
Let $(A,\al)$ and $(B,\be)$ be multivariable dynamical systems and assume that there exist associated operator algebras $\alg(A,\al)$ and  $\alg (B, \be)$ which are isometrically isomorphic. Does it follow that the multivariable systems $(A,\al)$ and $(B,\be)$  are piecewise conjugate over their \textit{Jacobson} spectra?
\end{question}

\begin{acknow}
The second author would like to thank Ken Davidson for inviting him to visit the University of Waterloo during the summer of 2012. The initial steps in this investigation were taken during that visit.
\end{acknow}


\end{document}